\documentclass[a4paper,12pt, reqno]{article}
\usepackage{amsthm}
\usepackage{amsmath,amssymb,latexsym,amsfonts,mathrsfs,bm}
\usepackage{color}
\usepackage[dvipdfmx]{graphicx}
\usepackage{bbm}
\usepackage{hyperref}
\usepackage{enumitem}
\usepackage{mathtools} 
\mathtoolsset{showonlyrefs=true}

\newcommand{\bC}{\ensuremath{\mathbb{C}}}

\newcommand{\bE}{\ensuremath{\mathbb{E}}}

\newcommand{\bN}{\ensuremath{\mathbb{N}}}

\newcommand{\bP}{\ensuremath{\mathbb{P}}}

\newcommand{\bR}{\ensuremath{\mathbb{R}}}

\newcommand{\cF}{\ensuremath{\mathcal{F}}}

\theoremstyle{plain}
\newtheorem{Thm}{Theorem}[section]

\newtheorem{Lem}[Thm]{Lemma}

\theoremstyle{definition}

\newtheorem{Rem}[Thm]{Remark}
\newtheorem{Ex}[Thm]{Example}

\numberwithin{equation}{section}

\begin{document}

\begin{center}
{\Large \bf 
Higher order approximations in arcsine laws for subordinators
}
\end{center}
\begin{center}
Toru Sera\footnote{Department of Mathematics, Graduate School of Science, Osaka University. Toyonaka, Osaka 560-0043, JAPAN.}
\end{center}


\begin{abstract}
We establish higher order approximations in the
Dynkin--Lamperti theorem,  a limit theorem for the distribution of a killed subordinator immediately before its first passage time over a fixed level. 
For this purpose,
we also study asymptotic expansions of potential densities for killed subordinators. 
\end{abstract}





\section{Introduction}

L\'evy's arcsine law for a one-dimensional Brownian motion is now classical in probability theory, and can be generalized to a variety of distributional identities and limit theorems. Among these generalizations, we focus on the Dynkin--Lamperti theorem for killed subordinators \cite[Theorem III.6]{Ber}, \cite[Theorem 5.16]{Kyp}. 
In this paper, we study asymptotic expansions of potential densities for killed subordinators, and then use them to obtain higher order approximations in the Dynkin--Lamperti theorem. 
We are motivated by similar approximations in operator renewal theory \cite{MelTer12, MelTer13, Ter15, Ter16}. The studies of convergence rates in other types of arcsine laws can be found in \cite{GolRei, Dob15, BorShe} for example.
The studies of potential densities for special subordinators can also be found in the context of potential analysis for subordinate Brownian motions
\cite{SonVon06, SikSonVon, BBKRSV, SchSonVon}.
We also remark that \cite{ChaKypSav} and \cite{DorSav} obtained series representations of potential densities for subordinators with strictly positive drift. Nevertheless, our asymptotic expansions of potential densities can also be applied to subordinators without drift.

Before going into the detail, let us briefly recall the Dynkin--Lamperti theorem.
Let $X=(X_t)_{t\geq0}$ be a subordinator, that is, a one-dimensional, non-decreasing L\'evy process with $\bP(X_0=0)=1$. We always assume that $X$ is not identically zero. We write $\Phi(\lambda)$ $(\lambda\geq0)$ for the Laplace exponent of $X$, that is, $\bE[\exp(-\lambda X_t)]= \exp(-\Phi(\lambda)t)$ ($\lambda,t\geq0$).
For $s>0$, let  $T(s)=\inf\{t\geq0\mid X_t>s\}$ denote the first passage time of $X$ strictly above $s$.
Then the following two conditions are equivalent for $0<\alpha<1$:
\begin{enumerate}[label=\upshape(\arabic*)]
\item\label{Intro:RV} For any $\lambda>0$, 
\begin{align}
\frac{\Phi(\lambda s)}{\Phi(s)}\to \lambda^\alpha, \quad\text{as $s\to 0+$}.
\end{align}
\item\label{Intro:arcsine} For any $0\leq t\leq 1$,
\begin{align}
    \bP\bigg(\frac{X_{T(s)-}}{s}\leq t\bigg)
    \to
    \frac{\sin(\pi\alpha)}{\pi}
    \int_0^t
    x^{-1+\alpha}(1-x)^{-\alpha}\,dx,
    \quad
    \text{as $s\to+\infty$.}	
\end{align}
\end{enumerate}
The limit distribution in \ref{Intro:arcsine} is the beta distribution with parameter $(\alpha,1-\alpha)$. In the case $\alpha=1/2$, it is nothing else but the usual arcsine distribution. If $X$ is an $\alpha$-stable subordinator, then both sides in \ref{Intro:RV} (respectively, in \ref{Intro:arcsine}) are equal for any $s>0$ without limit.

Let us illustrate the main result of this paper by giving examples. 
We are interested in convergence rates and higher order approximations in \ref{Intro:arcsine}.
By using our main results, we can obtain the following estimates. 
 
\begin{Ex}
Let us consider the case $\Phi(\lambda)=\lambda^\alpha+\lambda^\beta$ $(\lambda\geq0)$ for some $0<\alpha<\beta<1$. The corresponding subordinator $X$ is the independent sum of an $\alpha$-stable subordinator and a $\beta$-stable subordinator.	We will see that $\bP(X_{T(s)-}/s\in dx)$ admits a density which is continuous on $(0,1)$ and can be estimated as
\begin{align}
&\frac{\bP(X_{T(s)-}/s\in dx)}
     {dx}=
	\frac{\sin(\pi \alpha)}{\pi}
    x^{-1+\alpha}(1-x)^{-\alpha}
    +O(s^{-(\beta-\alpha)}),
    \quad
    \text{as $s\to+\infty$,}
    \label{Intro:double-stable}
\end{align}
uniformly in $x\in I$ for any closed interval $I\subset(0,1)$.
See Example \ref{Ex2:double-stable-infty} for a higher order approximation. Moreover, we can also obtain an explicit formula and an asymptotic expansion of the potential density, as discussed in Example \ref{Ex:double-stable-infty} and Remark \ref{Rem:double-stable}. We remark that the explicit formula of the potential density can also be found in \cite[Example 5.25]{Kyp}.
\end{Ex}

\begin{Ex}
	Let us consider the case $\Phi(\lambda)=\log(1+\lambda^\alpha)$ for some $0<\alpha<1$. The corresponding subordinator $X$ is called a geometric $\alpha$-stable subordinator.
	We will see that $\bP(X_{T(s)-}/s\in dx)$ admits a density which is continuous on $(0,1)$ and can be estimated as
\begin{align}
&\frac{\bP(X_{T(s)-}/s\in dx)}
     {dx}=
	\frac{\sin(\pi \alpha)}{\pi}
    x^{-1+\alpha}(1-x)^{-\alpha}
    +O(s^{-\alpha}),
    \quad
    \text{as $s\to+\infty$,}
    \label{Intro:geometric}
\end{align}
uniformly in $x\in I$ for any closed interval $I\subset(0,1)$.
See Example \ref{Ex2:geometric} for a higher order approximation. In this case, we do not know an explicit formula of the potential density in closed form. Nevertheless, we can obtain an asymptotic expansion  of the potential density, as discussed in Example \ref{Ex:geometric}.
\end{Ex}

The rest of this paper is organized as follows. In Section \ref{sec:pre}, we recall and summarize the Dynkin--Lamperti theorem for killed subordinators.
In Section \ref{sec:main}, we state our main results, that is, asymptotic expansions of potential densities and higher order approximations in the Dynkin--Lamperti theorem. Section \ref{sec:Proof-infty} is devoted to prove the main results. In Section \ref{sec:Ex}, we apply the main results to specific killed subordinators.

\section{Preliminaries}\label{sec:pre}

In this section, let us recall basic notions and properties of killed subordinators. 
We refer the reader to \cite{Ber}, \cite{Sat} and \cite{Kyp} for the details.


Let $X=(X_t)_{t\geq0}$ be a killed subordinator defined on a probability space $(\Omega,\cF,\bP)$.
In other words, there is a one-dimensional, non-decreasing L\'evy process $Y$ starting at the origin, and there is a $(0,+\infty]$-valued, exponentially-distributed random variable $\zeta$ such that $Y$ and $\zeta$ are independent, and
\begin{align}
  X_t
  =
  \begin{cases}
  Y_t, &t<\zeta,
  \\
  \partial, &t\geq\zeta, 	
  \end{cases}	
\end{align}
where $\partial$ denotes a cemetery point. 
Then the Laplace transform of $X$ can be represented as
\begin{align}\label{Laplace-trans}
   \bE[\exp(-z X_t)]
   =
   \exp(-\Phi(z)t),
   \quad 
   \text{$z\in \bC$ with ${\rm Re}\,z\geq0$, $t\geq0$.}	
\end{align}
Here the function $\Phi(z)$ is called the Laplace exponent of $X$, and it can be expressed in the form
\begin{align}\label{Laplace-exp}
   \Phi(z)
   =
   a+bz+\int_{(0,+\infty)}(1-e^{-z x})\,\Pi(dx),
   \quad 
   \text{$z\in \bC$ with ${\rm Re}\,z\geq0$},	
\end{align}
where $a,b\in[0,+\infty)$ are constants and $\Pi(dx)$ is a $\sigma$-finite measure on $(0,+\infty)$ satisfying
\begin{align}\label{Levy-meas}
   \int_{(0,+\infty)}(1\wedge x)\,\Pi(dx)<+\infty.	
\end{align}
The constant $a$ is called the killing rate, $b$ the drift coefficient, and the measure $\Pi(dx)$ the L\'evy measure of $X$. 
They are uniquely determined by the law of $X$.
Conversely, for any $a, b\in[0,+\infty)$ and any $\sigma$-finite measure $\Pi(dx)$ on $(0,+\infty)$ with \eqref{Levy-meas}, there exists a killed subordinator $X$ satisfying \eqref{Laplace-trans} and \eqref{Laplace-exp}. We say that $X$ is a subordinator if the killing rate is zero.

In the following we always assume $X$ is not identically zero, that is, $a>0$ or $b>0$ or $\Pi((0,+\infty))>0$. Then ${\rm Re}\,\Phi(z)>0$ for any $z\in\bC$ with ${\rm Re}\, z>0$.
The potential measure $U(dx)$ of $X$ is the locally finite measure on $[0,+\infty)$ defined as
\begin{align}
  U(A)
  =\bE\bigg[\int_0^{+\infty} \mathbbm{1}_{\{X_t\in A\}}\,dt\bigg],
  \quad
  \text{for any Borel set $A\subset [0,+\infty)$.}	
\end{align}
The Laplace transform of $U(dx)$ is expressed in the form
\begin{align}\label{Laplace-U}
   \int_{[0,+\infty)} e^{-zx}\,U(dx)
   =
   \frac{1}{\Phi(z)},
   \quad
   \text{$z\in\bC$ with ${\rm Re}\,z>0$}.  	
\end{align}
Hence $U(\{0\})=0$ if and only if $\lim_{\lambda\to+\infty}\Phi(\lambda)=+\infty$.
By differentiating both sides $N$ times and multiplying by $(-1)^N$, we have 
\begin{align}\label{Laplace-U-differential}
	\int_{[0,+\infty)} x^N e^{-zx}\,U(dx)
	=
	\bigg(-\frac{d}{dz}\bigg)^N \frac{1}{\Phi(z)},
	\quad
   \text{$z\in\bC$ with ${\rm Re}\,z>0$, $N\in\bN$}, 
\end{align}
where $\bN=\{1,2,\ldots\}$ denotes the set of positive integers.
The equality \eqref{Laplace-U-differential} plays a significant role in this paper.

For $s>0$, let $T(s)$ be the first passage time of $X$ strictly above $s$, that is,
\begin{align}
   T(s)=\inf\{t\geq0\mid X_t>s\}.	
\end{align}
We denote by $X_{t-}=\lim_{u\to t-} X_u$ the left limit of $X$ at $t\geq 0$. Here it is understood that $X_{0-}=X_0=0$.
Then
\begin{align}
    \bP(X_{T(s)-}\in dx)
    &
    \;(=\bP(X_{T(s)-}\in dx, \; T(s)<\zeta))
    \\
    &=
    \Pi((s-x,+\infty))\,U(dx),
    \quad 0\leq x<s.
    \label{XTy-density}	
\end{align}
See \cite[Proposition III.2]{Ber} and \cite[Theorem 5.6]{Kyp} for the details. The equality \eqref{XTy-density} remains valid for $x=s$ if the drift coefficient $b$ is zero.
If $b>0$, then there exists a potential density $u(x)=U(dx)/dx$ which is positive and continuous on $(0,+\infty)$, and
\begin{align}
	\bP(X_{T(s)-}=s)=\bP(X_{T(s)}=s)=bu(s),
\end{align}
as shown in \cite[Theorems III.4 and III.5]{Ber} and \cite[Theorem 5.9]{Kyp}.

A measurable function $f:(0,+\infty)\to(0,+\infty)$ is said to be regularly varying at $+\infty$ (respectively, at $0+$) with index $\rho\in\bR$ if
\begin{align}
    \lim_{x\to+\infty}\frac{f(\lambda x)}{f(x)}=\lambda^\rho
    \quad\bigg(\text{respectively,} \,\,\lim_{x\to0+}\frac{f(\lambda x)}{f(x)}=\lambda^\rho\bigg),
    \quad 
    \text{for any $\lambda>0$.}	
\end{align}
See \cite{BinGolTeu} for the detail of the theory of regular variation.

Let us recall the Dynkin--Lamperti theorem for killed subordinators. The proof can be found in \cite[Section III.3]{Ber} and \cite[Section 5.5]{Kyp}.

\begin{Thm}[long-range Dynkin--Lamperti theorem]\label{Thm:DL-infty}
Let $\alpha\in(0,1)$. Let $X=(X_t)_{t\geq0}$ be a killed subordinator which is not identically zero. Then the following conditions are equivalent:
\begin{enumerate}[label={\upshape (\arabic*)}]

\item The killing rate $a$ is zero, and the tail $\Pi((x,+\infty))$ of the L\'evy measure  is regularly varying at $+\infty$ with index $-\alpha$.

\item\label{Lap-RV} The Laplace exponent $\Phi:(0,+\infty)\to(0,+\infty)$ (restricted over $(0,+\infty)$) is regularly varying at $0+$ with index $\alpha$. 

\item\label{Thm:DL-infty-1}	
$\displaystyle \lim_{s\to+\infty}\bP\bigg(\frac{X_{T(s)-}}{s} \leq t\bigg)=\frac{\sin(\pi\alpha)}{\pi}\int_0^t x^{-1+\alpha}(1-x)^{-\alpha}\,dx$ for any $0\leq t\leq 1$.

\item 
$\displaystyle \lim_{s\to+\infty} s^{-1}\bE[X_{T(s)-}]=\alpha$.

\item The law of $tX_{1/\Phi(t)}$ converges weakly to a stable law with index $\alpha$, as $t\to0+$. 

\item\label{U-RV} The distribution function $U([0,x])$ of the potential measure  is regularly varying at $+\infty$ with index $\alpha$.  
\end{enumerate}
\end{Thm}

We remark that the equivalence between the conditions \ref{Lap-RV} and \ref{U-RV} follows from \eqref{Laplace-U} and Karamata's Tauberian theorem \cite[Theorems 1.7.1 and 1.7.1']{BinGolTeu}. The following theorem is a short-range version of Theorem \ref{Thm:DL-infty}.

\begin{Thm}[short-range Dynkin--Lamperti theorem]\label{Thm:DL-zero}	
Let $\beta\in(0,1)$. 
Let $X$ be a killed subordinator which is not identically zero.
Then the following conditions are equivalent:
\begin{enumerate}[label={\upshape (\arabic*')}]
\item The drift coefficient $b$ is zero, and the tail $\Pi((x,+\infty))$ of the L\'evy measure  is regularly varying at $0+$ with index $-\beta$.

\item The Laplace exponent $\Phi:(0,+\infty)\to(0,+\infty)$ (restricted over $(0,+\infty)$) is regularly varying at $+\infty$ with index $\beta$. 

\item 	
$\displaystyle \lim_{s\to0+}
\bP\bigg(\frac{X_{T(s)-}}{s} \leq t\bigg)=\frac{\sin(\pi\beta)}{\pi}\int_0^t x^{-1+\beta}(1-x)^{-\beta}\,dx$ for any $0\leq t\leq 1$.
\vspace{3pt}
\item 
$\displaystyle \lim_{s\to0+} s^{-1}\bE[X_{T(s)-}]=\beta$.

\item The law of $t X_{1/\Phi(t)}$ converges weakly to a stable law with index $\beta$, as $t\to+\infty$. 

\item The distribution function $U([0,x])$ of the potential measure  is regularly varying at $0+$ with index $\beta$.  
\end{enumerate}
\end{Thm}

\begin{Rem}[regular variation of potential density for special subordinator]
If we further assume that $\Phi:(0,+\infty)\to(0,+\infty)$ is a special Bernstein function in the sense of \cite{SonVon06, BBKRSV, SchSonVon}, then the potential measure $U|_{(0,+\infty)}(dx)$ restricted over $(0,+\infty)$ admits a non-increasing density $u(x)$ $(0<x<+\infty)$. By the monotone density theorem \cite[Theorems 1.7.2 and 1.7.2b]{BinGolTeu}, any one of conditions in Theorem \ref{Thm:DL-infty} is also equivalent to
\begin{enumerate}[label={\upshape (\arabic*)}]
\setcounter{enumi}{6}
\item $u(x)$ is regularly varying at $+\infty$ with index $-1+\alpha$. 	
\end{enumerate}
Similarly, if $\Phi$ is a special Bernstein function, then any one of conditions in Theorem \ref{Thm:DL-zero} is also equivalent to
\begin{enumerate}[label={\upshape (\arabic*')}]
\setcounter{enumi}{6}
\item $U(\{0\})=0$ and $u(x)$ is regularly varying at $0+$ with index $-1+\beta$.
\end{enumerate}
Nevertheless, we will not assume $\Phi$ is a special Bernstein function in our abstract main results.	
\end{Rem}

\section{Main results}\label{sec:main}

Let $f(x), g(x), h(x)$ be functions with $h(x)>0$.
We write $f(x)=g(x)+O(h(x))$ as $x\to a$ if there exists some constant $C\in(0,+\infty)$ such that $|f(x)-g(x)|\leq C h(x)$ on some neighborhood of $a$. We write $f(x)=g(x)+o(h(x))$ as $x\to a$ if $|f(x)-g(x)|/h(x)\to 0$ as $x\to a$. We write $f(x)\sim h(x)$ as $x\to a$ if $f(x)/h(x)\to 1$ as $x\to a$.

In this paper, the functions $\log z$ and $z^w$ ($z\in\bC\setminus(-\infty,0]$, $w\in \bC$) always take the principal values, that is, ${\rm Im}\, \log z\in(-\pi,\pi)$ and $z^w=\exp(w \log z)$.
We denote by
\begin{align}
  \frac{1}{\Gamma(z)}
  =
  \lim_{n\to\infty}\frac{z(z+1)(z+2)\cdots (z+n)}{n! n^z},
  \quad
  z\in\bC,	
\end{align}
the reciprocal gamma function. Note that $1/\Gamma(-n)=0$ for all $n\in\bN\cup\{0\}$.

We now explain our main results. 

\begin{Thm}[asymptotic expansion of potential density at $+\infty$]\label{Thm:potential-infty}
Let $X=(X_t)_{t\geq0}$ be a killed subordinator which is not identically zero. Let $N\in\bN$ be fixed. Assume that the Laplace exponent $\Phi(z)$ satisfies the following two conditions:
\begin{enumerate}[label={\upshape (\Alph*)}]
\item\label{condition:zero} There exist constants $n\in\bN\cup\{0\}$, $c_0\in(0,+\infty)$, $c_1,\dots,c_n\in\bR\setminus\{0\}$ and  $ \alpha_0>\alpha_1>\dots>\alpha_n>-N+1$ with $\alpha_0\in[0,1]$, and a locally integrable function $R:(0,+\infty)\to(0,+\infty)$ such that $R(t)$ is non-increasing on $(0, r)$ for some $r>0$,  $R(t)=o(t^{-N-\alpha_n})$, as $t\to0+$, and
\begin{align}
    \bigg(\frac{d}{dz}\bigg)^{N}
     \bigg(
    \frac{1}{\Phi(z)}
    -
    \sum_{k=0}^n\frac{1}{c_k z^{\alpha_k}}
    \bigg)
    =
    O(R(|z|)),
    \quad
    \text{as $z\to0$ in $\{z\in\bC\mid {\rm Re}\,z>0\}$}.	
\end{align} 

\item\label{condition:bdd}  For any $r>0$, 
\begin{align}
   \sup_{\lambda>0} \int_{\bR\setminus(-r,r)} 
   \bigg|
   \bigg(\bigg(\frac{d}{dz}\bigg)^N
   \frac{1}{\Phi}\bigg)(\lambda+i\theta)\bigg|\,d\theta<+\infty.	
\end{align}
\end{enumerate}
Then, the potential measure $U|_{(0,+\infty)}(dx)$ restricted over $(0,+\infty)$ admits a continuous density $u(x)$ $(0<x<+\infty)$ which can be estimated as
\begin{align}
   u(x)
   =
   \sum_{k=0}^n\frac{x^{-1+\alpha_k}}{c_k\Gamma(\alpha_k)}+O(\varepsilon_1(x)),
   \quad\text{as $x\to+\infty$},
\end{align}
where 
\begin{align}
   \varepsilon_1(x)=x^{-N-1}R(x^{-1})+x^{-N}\int_{x^{-1}}^{1} R(t)\,dt.
\end{align}
\end{Thm}

The proof of Theorem \ref{Thm:potential-infty} will be given in Section \ref{sec:Proof-infty}.

\begin{Rem}\label{Rem:Pi-infty}
The condition \ref{condition:zero} with $\alpha_0=\alpha\in(0,1]$ yields that
\begin{align}
	\Phi(\lambda)\sim 
	c_0\lambda^\alpha,
	\quad
	\text{as $\lambda\to 0+$,} 
\end{align}
and hence the killing rate must be zero.
In addition, if $\alpha\in(0,1)$, then we use Karamata's Tauberian theorem and the monotone density theorem to obtain
\begin{align}
   \Pi((x,+\infty))\sim \frac{c_0x^{-\alpha}}{\Gamma(1-\alpha)},
   \quad
   \text{as $x\to+\infty$}.	
\end{align}
\end{Rem}

\begin{Rem}
For example, if $R(t)=t^{-\gamma}$ for some $\gamma\in [0,N+\alpha_n)$, then
\begin{align}
    \varepsilon_1(x) 
    =
   \begin{cases}
   	O(x^{-N-1+\gamma}), &1<\gamma<N+\alpha_n,
   	\\[3pt]
   	O(x^{-N}\log x),
   	& \gamma=1,
   	\\[3pt]
   	O(x^{-N}),
   	&0\leq\gamma<1,
   \end{cases}
   \quad \text{as $x\to+\infty$.}
\end{align}	
\end{Rem}

Combining Theorem \ref{Thm:potential-infty} with the equality \eqref{XTy-density} and Euler's reflection formula, we obtain the following theorem immediately.

\begin{Thm}[higher order approximation in long-range Dynkin--Lamperti theorem]\label{Thm:conv-DL-infty}
Under the assumptions of Theorem $\ref{Thm:potential-infty}$ with  $\alpha_0=\alpha\in[0,1]$, let us define a function $\varepsilon_2:(0,+\infty)\to\bR$ by
\begin{align}
    \Pi((x,+\infty))=
    \mathbbm{1}_{\{\alpha>0\}}
    \frac{c_0 x^{-\alpha}}{\Gamma(1-\alpha)}+\varepsilon_2(x),
    \quad 0<x<+\infty.	
\end{align}
Set
\begin{align}
   \varepsilon_3(s,x)
   =
   s\Pi((s-sx,+\infty))\varepsilon_1(sx),
   \quad 0<s, x<+\infty.	
\end{align}
Then, for $s\in(0,+\infty)$, the finite measure 
\begin{align}
\bP\bigg(\frac{X_{T(s)-}}{s}\in dx,\; 0<\frac{X_{T(s)-}}{s}<1\bigg),
\quad
0<x<1,
\end{align}
 admits a density expressed in the form
\begin{align}
   s\Pi(&(s-sx,+\infty))u(sx)
   \\
   =\,&\frac{\sin(\pi \alpha)}{\pi}x^{-1+\alpha}(1-x)^{-\alpha}
   +
    \mathbbm{1}_{\{\alpha>0\}}
   \sum_{k=1}^n
   \frac{c_0 s^{\alpha_k-\alpha}}{c_k\Gamma(\alpha_k)\Gamma(1-\alpha)}x^{-1+\alpha_k}(1-x)^{-\alpha}
   \\
   &+\sum_{k=0}^n\frac{s^{\alpha_k}}{c_k\Gamma(\alpha_k)}
   x^{-1+\alpha_k}\varepsilon_2(s-sx)
   +
   O(\varepsilon_3(s,x)),
   \quad
   \text{as $s\to+\infty$,}
   \label{Thm:conv-DL-infty-1}
\end{align}
uniformly in $x\in I$ for any closed interval $I\subset (0,1)$. 
Therefore, for any closed interval $I\subset (0,1)$,
\begin{align}
	\bP\bigg(&\frac{X_{T(s)-}}{s}\in I\bigg)
	\\
	=
	\,&\frac{\sin(\pi \alpha)}{\pi}\int_I x^{-1+\alpha}(1-x)^{-\alpha}\,dx
	\\
	&+
	 \mathbbm{1}_{\{\alpha>0\}}
	\sum_{k=1}^n
   \frac{c_0 s^{\alpha_k-\alpha}}{c_k\Gamma(\alpha_k)\Gamma(1-\alpha)}
   \int_I x^{-1+\alpha_k}(1-x)^{-\alpha}\,dx
   \\
   &+\sum_{k=0}^n\frac{s^{\alpha_k}}{c_k\Gamma(\alpha_k)}
   \int_I
   x^{-1+\alpha_k}\varepsilon_2(s-sx)\,dx
   +
  O\bigg(\sup_{x\in I}\varepsilon_3(s,x)\bigg),
   \quad
   \text{as $s\to+\infty$}.
   \label{Thm:conv-DL-infty-2}
\end{align}

\begin{Rem}
By Remark \ref{Rem:Pi-infty}, if $0<\alpha<1$, then $\varepsilon_2(x)=o(x^{-\alpha})$, as $x\to+\infty$.	
\end{Rem}

\end{Thm}

Similarly, we can also obtain short-range versions of Theorems \ref{Thm:potential-infty} and \ref{Thm:conv-DL-infty}, as we shall see below.

\begin{Thm}[asymptotic expansion of potential density at $0+$]\label{Thm:potential-zero}
Let $X=(X_t)_{t\geq0}$ be a killed subordinator which is not identically zero. Assume that the Laplace exponent $\Phi(z)$ satisfies the following condition:
\begin{enumerate}[label={\upshape (\Alph*')}]
\item\label{condition:infty} There exist constants $n\in\bN\cup\{0\}$, $c_0\in(0,+\infty)$, $c_1,\dots,c_n\in\bR\setminus\{0\}$ and $\beta_0<\beta_1<\dots<\beta_n<+\infty$ with $\beta_0\in(0,1]$, and a locally integrable function $R:(0,+\infty)\to(0,+\infty)$ such that $R(t)$ is non-increasing on $(r,+\infty)$ for some $r>0$, $R(t)=o(t^{-1-\beta_n})$, as $t\to+\infty$, and
\begin{align}
    \frac{d}{dz}
    \bigg(
    \frac{1}{\Phi(z)}
    -
    \sum_{k=0}^n
    \frac{1}{c_k z^{\beta_k}}
    \bigg)
    =
    O(R(|z|)),
    \quad
    \text{as $z\to\infty$ in $\{z\in\bC\mid {\rm Re}\,z>0\}$}.	
\end{align}
\end{enumerate}
Then the potential measure $U|_{(0,+\infty)}$ restricted over $(0,+\infty)$ admits a continuous density $u(x)$ $(0<x<+\infty)$ which can be estimated as
\begin{align}
   u(x)
   =
   \sum_{k=0}^n
   \frac{x^{-1+\beta_k}}{c_k\Gamma(\beta_k)}+O(\varepsilon_1(x)),
   \quad\text{as $x\to0+$,}
\end{align}
where
\begin{align}
    \varepsilon_1(x)
    =
    x^{-2}R(x^{-1})+x^{-1}\int_{x^{-1}}^{+\infty} R(t)\,dt.	
\end{align}
\end{Thm}

We will prove Theorem \ref{Thm:potential-zero} in Section \ref{sec:Proof-infty}.

\begin{Rem}\label{Rem:Pi-zero}
The condition \ref{condition:infty} with $\beta_0=\beta\in(0,1]$ yields that
\begin{align}
	\Phi(\lambda)\sim 
	c_0\lambda^\beta,
	\quad
	\text{as $\lambda\to +\infty$}. 
\end{align}
Let us further assume $\beta\in(0,1)$. Then the drift coefficient must be zero. In addition, we use Karamata's Tauberian theorem and the monotone density theorem to obtain
\begin{align}
   \Pi((x,+\infty))\sim \frac{c_0x^{-\beta}}{\Gamma(1-\beta)},
   \quad
   \text{as $x\to 0+$}.	
\end{align}
\end{Rem}

\begin{Rem}
For example, if $R(t)=t^{-\gamma}$ for some $\gamma\in(1+\beta_n,+\infty)$, then
\begin{align}
    \varepsilon_1(x)=O(x^{-2+\gamma}), 
    \quad \text{as $x\to0+$.}	
\end{align}
\end{Rem}

The following theorem immediately follows from Theorem \ref{Thm:potential-zero} and the equality \eqref{XTy-density}.

\begin{Thm}[higher order approximation in short-range Dynkin--Lamperti theorem]
\label{Thm:conv-DL-zero}
Under the assumption of Theorem $\ref{Thm:potential-zero}$ with $\beta_0=\beta\in(0,1]$,  let us define a function $\varepsilon_2:(0,+\infty)\to\bR$ by 
\begin{align}
    \Pi((x,+\infty))=
    \frac{c_0 x^{-\beta}}{\Gamma(1-\beta)}+\varepsilon_2(x),
    \quad
   0<x<+\infty.	
\end{align}
Set 
\begin{align}
 \varepsilon_3(s,x)
   =
   s\Pi((s-sx,+\infty))\varepsilon_1(sx),
   \quad 0<s, x<+\infty.	
\end{align}
Then, for $s\in(0,+\infty)$, the finite measure 
\begin{align}
\bP\bigg(\frac{X_{T(s)-}}{s}\in dx,\; 0<\frac{X_{T(s)-}}{s}<1\bigg),
\quad 0<x<1,
\end{align}
 admits a density expressed in the form
\begin{align}
   s\Pi((s-sx,&+\infty))u(sx)
   \\
   =\,&
   \frac{\sin(\pi\beta)}{\pi}x^{-1+\beta}(1-x)^{-\beta}
   +
   \sum_{k=1}^n \frac{c_0 s^{\beta_k-\beta}}{c_k\Gamma(\beta_k)\Gamma(1-\beta)}
   x^{-1+\beta_k}(1-x)^{-\beta}
   \\
   &+
   \sum_{k=0}^n \frac{s^{\beta_k}}{c_k\Gamma(\beta_k)}
   x^{-1+\beta_k}\varepsilon_2(s-sx)
   +
   O(\varepsilon_3(s,x)),
   \quad
   \text{as $s\to0+$,}	
\end{align}
uniformly in $x\in I$ for any closed interval $I\subset (0,1)$. 
Therefore, for any closed interval $I\subset (0,1)$,
\begin{align}
	&\bP\bigg(\frac{X_{T(s)-}}{s}\in I\bigg)
	\\
	=\,&
	 \frac{\sin(\pi\beta)}{\pi}
	 \int_I x^{-1+\beta}(1-x)^{-\beta}\,dx
	 +
	 \sum_{k=1}^n \frac{c_0 s^{\beta_k-\beta}}{c_k\Gamma(\beta_k)\Gamma(1-\beta)}
   \int_I x^{-1+\beta_k}(1-x)^{-\beta}\,dx 
   \\	
   &+
   \sum_{k=0}^n \frac{s^{\beta_k}}{c_k\Gamma(\beta_k)}
   \int_I
   x^{-1+\beta_k}\varepsilon_2(s-sx)\,dx
   + 
   O\bigg(\sup_{x\in I}\varepsilon_3(s,x)\bigg),
   \quad
   \text{as $s\to0+$}.
\end{align}
\end{Thm}

\begin{Rem}
By Remark \ref{Rem:Pi-zero}, if $0<\beta<1$, then $\varepsilon_2(x)=o(x^{-\beta})$, as $x\to0+$.	
\end{Rem}

\section{Proofs of Theorems \ref{Thm:potential-infty} and \ref{Thm:potential-zero}} \label{sec:Proof-infty}

We imitate the method of \cite{Ter16}.
First, we will give a representation formula of a potential density in terms of the Laplace exponent.

\begin{Lem}\label{Lem:potential}
Let $X=(X_t)_{t\geq0}$ be a killed subordinator which is not identically zero. Let $N\in\bN$ and $\lambda\in(0,+\infty)$ be fixed. Assume that the Laplace exponent $\Phi(z)$ satisfies
\begin{align}\label{condition:integrable}
    \int_{-\infty}^{+\infty}
    \bigg|
    \bigg(
    \bigg(\frac{d}{dz}\bigg)^N
    \frac{1}{\Phi}\bigg)
    (\lambda+i\theta)
    \bigg|\,d\theta<+\infty.	
\end{align}
Then the potential measure $U|_{(0,+\infty)}$ restricted over $(0,+\infty)$ admits a continuous density 
\begin{align}\label{u(x)}
   u(x)=\frac{x^{-N}e^{\lambda x}}{2\pi}\int_{-\infty}^{+\infty}
       e^{i\theta x}
       \bigg(
      \bigg(-\frac{d}{dz}\bigg)^N \frac{1}{\Phi}\bigg)
    (\lambda+i\theta)
        \,d\theta,
       \quad
       0<x<+\infty.	
\end{align}
\end{Lem}

\begin{proof}
By substituting $z=\lambda+i\theta$ into \eqref{Laplace-U-differential}, we have
\begin{align}
	\int_{[0,+\infty)} x^N e^{-\lambda x-i\theta x}\,U(dx)
	=
	\bigg(
	\bigg(-\frac{d}{dz}\bigg)^N
	\frac{1}{\Phi}\bigg)
    (\lambda+i\theta),
    \quad
    \theta\in\bR.
\end{align}
By the inversion formula (see for example \cite[Theorem 3.3.14]{Dur}), we obtain the desired result.	
\end{proof}

We now prove Theorem \ref{Thm:potential-infty} by using Lemma \ref{Lem:potential}.

\begin{proof}[Proof of Theorem \upshape\ref{Thm:potential-infty}]
Recall that the Fourier transforms of gamma distributions can be calculated as
\begin{align}
   \frac{1}{\Gamma(1+\alpha)}
   \int_0^{+\infty}
   e^{-i\theta x}
   e^{-\lambda x}
   x^{\alpha}\,dx
   =
   \frac{1}{(\lambda+i\theta)^{1+\alpha}},
   \quad
   \alpha,\lambda\in(0,+\infty),\;
   \theta\in\bR.	
\end{align}
By the inversion formula,
\begin{align}\label{Proof-0}
   \frac{1}{\Gamma(1+\alpha)}
   e^{-\lambda x}x^\alpha
   =
   \frac{1}{2\pi}\int_{-\infty}^{+\infty} \frac{e^{i\theta x}}{(\lambda+i\theta)^{1+\alpha}}\,d\theta,
   \quad
   \alpha, \lambda ,x \in(0,+\infty).	
\end{align}
Substituting $\lambda=x^{-1}$ and $\alpha=N-1+\alpha_k$ into \eqref{Proof-0}, we have
\begin{align}
	 \frac{1}{\Gamma(N+\alpha_k)}
   e^{-1}x^{N-1+\alpha_k}
   =
   \frac{1}{2\pi}
   \int_{-\infty}^{+\infty}
   \frac{e^{i\theta x}}{(x^{-1}+i\theta)^{N+\alpha_k}}\,d\theta,
   \quad
   0<x<+\infty,
\end{align}
which implies
\begin{align}
&\frac{1}{\Gamma(\alpha_k)}x^{-1+\alpha_k}
\\
&=
\frac{x^{-N} e}{2\pi}
\int_{-\infty}^{+\infty}
e^{i\theta x}\frac{\alpha_k(1+\alpha_k)\cdots(N-1+\alpha_k)}{(x^{-1}+i\theta)^{N+\alpha_k}}\,d\theta,
\quad
   0<x<+\infty.	
   \label{Proof-1}
\end{align}

By the assumption \ref{condition:zero}, we can take some constants $r>0$ and $C>0$ such that, $R(t)$ is non-increasing on $(0,\sqrt{2}r)$, and for any $z\in\bC$ with ${\rm Re}\,z>0$ and $|z|\leq \sqrt{2}r$, we have
\begin{align}
   \bigg|
  \bigg(-\frac{d}{dz}\bigg)^N \frac{1}{\Phi(z)}
   -
   \sum_{k=0}^n\frac{\alpha_k (1+\alpha_k)\cdots (N-1+\alpha_k) }{c_k z^{N+\alpha_k}}
   \bigg|
   \leq
   C R(|z|).
   \label{Proof-2}
\end{align}
By the assumption \ref{condition:bdd},
the condition \eqref{condition:integrable} is satisfied for any $\lambda>0$.
By Lemma \ref{Lem:potential}, we see that $U|_{(0,+\infty)}$ admits a continuous density $u(x)$ expressed in the form  
\eqref{u(x)}. Substituting $\lambda=x^{-1}$ into \eqref{u(x)}, we have
\begin{align}
   u(x)
   &=
  \frac{x^{-N}e}{2\pi}\int_{-\infty}^{+\infty}
       e^{i\theta x}
      \bigg(
      \bigg(-\frac{d}{dz}\bigg)^N 
      \frac{1}{\Phi}\bigg)
    (x^{-1}+i\theta)
        \,d\theta
   \\
   &=
   \frac{x^{-N}e}{2\pi}
   \int_{-r}^r e^{i\theta x}
      \bigg(
      \bigg(-\frac{d}{dz}\bigg)^N 
      \frac{1}{\Phi}\bigg)
    (x^{-1}+i\theta)\,d\theta
   +
   O(x^{-N}),
   \quad
   \text{as $x\to+\infty$}.
   \label{Proof-3}
\end{align}
Here we used \ref{condition:bdd} again.

Let $x\in(r^{-1},\infty)$.
Combining \eqref{Proof-1} with \eqref{Proof-2} and \eqref{Proof-3},
\begin{align}
  \bigg|u(x)-
  &\sum_{k=0}^n\frac{x^{-1+\alpha_k}}{c_k\Gamma(\alpha_k)}\bigg|
  \\
  =\,&
   \bigg|u(x)-\frac{x^{-N}e}{2\pi}
   \sum_{k=0}^n\int_{-\infty}^{+\infty} 
   e^{i\theta x}
   \frac{\alpha_k (1+\alpha_k)\cdots (N-1+\alpha_k)}
   {c_k(x^{-1}+i\theta)^{N+\alpha_k}}\, d\theta\bigg|
   \\
   \leq\,&
  \frac{x^{-N}e}{2\pi}\
   \sum_{k=0}^n
  \int_{\bR\setminus (-r,r)}
   \frac{|\alpha_k (1+\alpha_k)\cdots (N-1+\alpha_k)|}
   {|c_k|\cdot|x^{-1}+i\theta|^{N+\alpha_k}}
   \, d\theta
   \\
   &
   +
   \frac{x^{-N}e}{2\pi}\
   C\int_{-r}^r R(|x^{-1}+i\theta|)\,d\theta
   + O(x^{-N}),
   \quad
   \text{as $x\to+\infty$.}
   \label{Proof-4}
\end{align}
The first term of the right-hand side in \eqref{Proof-4} is $O(x^{-N})$, which follows from
\begin{align}
 \sup_{x\in(0,+\infty)}\int_{\bR\setminus (-r,r)}
   \frac{1}
   {|x^{-1}+i\theta|^{N+\alpha_k}}
   \, d\theta
   \leq
   \int_{\bR\setminus (-r,r)}\frac{d\theta}{|\theta|^{N+\alpha_k}}<+\infty.
\end{align}
Let us estimate the second term of the right-hand side in  \eqref{Proof-4}. Since $R(t)$ is non-increasing on $(0,\sqrt{2}r)$, we have
\begin{align}
\int_{-r}^r R(|x^{-1}+i\theta|)\,d\theta
&\leq
2\int_0^{x^{-1}}R(x^{-1})\,d\theta
+
2\int_{x^{-1}}^r R(|i\theta|)\,d\theta
\\
&=
2x^{-1}R(x^{-1})
+
2\int_{x^{-1}}^r R(t)\,dt.
\end{align}
Therefore we obtain the desired result.
\end{proof}

The proof of Theorem \ref{Thm:potential-zero} is almost the same as that of Theorem \ref{Thm:potential-infty}, as we shall see in the following.

\begin{proof}[Proof of Theorem \upshape\ref{Thm:potential-zero}]
By the assumption \ref{condition:infty}, there exist some constants $r>0$ and $C>0$ such that, $R(t)$ is non-increasing on $(r,+\infty)$, and for any $z\in\bC$ with ${\rm Re}\,z>0$ and $|z|>r$, we have
\begin{align}
    \bigg|
    \frac{\Phi'(z)}{\Phi(z)^2}
    -\sum_{k=0}^n\frac{\beta_k }{c_k z^{1+\beta_k}}
    \bigg|
    \leq
    C R(|z|).
    \label{Proof-11}	
\end{align}
Then the condition \eqref{condition:integrable} is satisfied for any $\lambda\in(0,+\infty)$ with $N=1$.
By Lemma \ref{Lem:potential}, we see that $U|_{(0,+\infty)}$ admits a continuous density $u(x)$ expressed in the form  
\eqref{u(x)}.

Let $x\in(0,r^{-1})$. Substituting $\lambda=x^{-1}$	into \eqref{u(x)} with $N=1$, we have
\begin{align}
u(x)
   &=
  \frac{x^{-1}e}{2\pi}\int_{-\infty}^{+\infty}
       e^{i\theta x}
       \frac{\Phi'(x^{-1}+i\theta)}{\Phi(x^{-1}+i\theta)^2}
        \,d\theta.
          \label{Proof-12}
\end{align}
By a similar calculation as in \eqref{Proof-4}, we see that
\begin{align}
 \bigg|u(x)-\sum_{k=0}^n\frac{x^{-1+\beta_k}}{c_k\Gamma(\beta_k)}\bigg|
 &=
   \bigg|u(x)
   -\frac{x^{-1}e}{2\pi}
   \sum_{k=0}^n
   \int_{-\infty}^{+\infty} e^{i\theta x}\frac{\beta_k }
   {c_k(x^{-1}+i\theta)^{1+\beta_k}}\, d\theta\bigg|
   \\
   &\leq
   \frac{x^{-1}e}{2\pi}
   C
   \int_{-\infty}^{+\infty}  R(|x^{-1}+i\theta|)\,d\theta.
   \label{Proof-13}	
\end{align}
Let us estimate the right-hand side of \eqref{Proof-13}. Since $R(t)$ is non-increasing on $(r,+\infty)$, we have
\begin{align}
   \int_{-\infty}^{+\infty} R(|x^{-1}+i\theta|)\,d\theta
   &\leq
   2\int_0^{x^{-1}}R(x^{-1})\,d\theta
   +
   2\int_{x^{-1}}^{+\infty} R(|i\theta|)\,d\theta
   \\
   &=
   2x^{-1}R(x^{-1})
   +
   2\int_{x^{-1}}^{+\infty} R(t)\,dt. 
   \label{Proof-14}	
\end{align}
We now complete the proof.
\end{proof}

\section{Examples}\label{sec:Ex}

We can construct a variety of specific examples of our abstract results by referencing \cite[Section 5.2.2]{BBKRSV} and \cite[Chapter 16]{SchSonVon}.
For example, we can apply our abstract results to the independent sum of stable subordinators, as we shall see in the next example.

\begin{Ex}
\label{Ex:double-stable-infty}
Let us consider the case 
\begin{align}
\Phi(z)=C_1z^\alpha+C_2 z^\beta
\end{align}
for some $C_1,C_2\in(0,+\infty)$ and $0\leq\alpha<\beta\leq 1$. We remark that this case was studied in \cite[Example 5.25]{Kyp}. If $0<\alpha<\beta<1$, then the corresponding subordinator $X$ is the independent sum of an $\alpha$-stable subordinator and a $\beta$-stable subordinator.
Let us check that the assumptions of Theorem \ref{Thm:potential-infty} are satisfied.
Note that
\begin{align}
  \frac{1}{\Phi(z)}
  =
  \frac{1}{C_1z^{\alpha}(1+C_2 z^{\beta-\alpha}/C_1)}
  =
  \frac{1}{C_1}\sum_{k=0}^\infty 
  \bigg(-\frac{C_2}{C_1}\bigg)^k z^{-\alpha+k(\beta-\alpha)},
\end{align}
for any $z\in\bC$ with ${\rm Re}\,z>0$ and $C_2|z|^{\beta-\alpha}/C_1<1$.
Fix $n\in\bN\cup\{0\}$ arbitrarily. Take $N\in\bN$ large enough so that $-N-\alpha+(n+1)(\beta-\alpha)<-1$. 
Then we have
\begin{align}
   \bigg(\frac{d}{dz}\bigg)^N
   \bigg(
   \frac{1}{\Phi(z)}
   -
    \frac{1}{C_1}\sum_{k=0}^n
  \bigg(-\frac{C_2}{C_1}\bigg)^k z^{-\alpha+k(\beta-\alpha)}
   \bigg)
   =
   O(|z|^{-N-\alpha+(n+1)(\beta-\alpha)}),
   \\
   \text{as $z\to 0$ in $\{z\in\bC\mid {\rm Re}\,z>0\}$.}	
\end{align}
Therefore the condition \ref{condition:zero} is satisfied with 
\begin{align}
	c_k=C_1\bigg(-\frac{C_1}{C_2}\bigg)^k,
	\quad
	\alpha_k=\alpha-k(\beta-\alpha)
	\quad
	\text{and}
	\quad
	R(t)=t^{-N-\alpha+(n+1)(\beta-\alpha)}.
\end{align}
It is easily seen that there exists a polynomial $p_N(z^\alpha,z^\beta)$ of degree at most $N$ in $z^\alpha$ and $z^\beta$ such that
\begin{align}
    \bigg(\frac{d}{dz}\bigg)^N \frac{1}{\Phi(z)}=\frac{p_N(z^\alpha,z^\beta)}{z^N(C_1 z^{\alpha}+C_2 z^{\beta})^{N+1}}.
\end{align}
Hence, for any $r>0$, there exists a constant $C=C(r,N)>0$ such that
\begin{align*}
    \bigg|\bigg(\frac{d}{dz}\bigg)^N \frac{1}{\Phi(z)}\bigg|
    \leq
    C|z|^{-N-\beta},
    \quad
    \text{on $\{z\in\bC\mid {\rm Re}\,z>0,\, |{\rm Im}\,z|>r\}$},	
\end{align*}
which implies
\begin{align}
		\sup_{\lambda>0}\int_{\bR\setminus(-r,r)}
   \bigg|
   \bigg(
   \bigg(\frac{d}{dz}\bigg)^N 
   \frac{1}{\Phi}\bigg)(\lambda+i\theta)\bigg|\,d\theta
   	\leq
   	C
   	\int_{\bR\setminus(-r,r)}
   	|\theta|^{-N-\beta}
   	\,d\theta
   	<+\infty.
\end{align}
Therefore the condition \ref{condition:bdd} is satisfied.
 By Theorem \ref{Thm:potential-infty}, we see that $U|_{(0,+\infty)}(dx)$ admits a density $u(x)$ $(0<x<+\infty)$ which is continuous on $(0, +\infty)$ and can be estimated as 
\begin{align}
    u(x)=\,&\frac{1}{C_1}\sum_{k=0}^n 
    \bigg(-\frac{C_2}{C_1}\bigg)^k
    \frac{x^{-1+\alpha-k(\beta-\alpha)}}{\Gamma(\alpha-k(\beta-\alpha))}
     \\
     &+O(x^{-1+\alpha-(n+1)(\beta-\alpha)}),
    \quad
    \text{as $x\to+\infty$.}
    \label{Ex:double-stable-infty-1}	
\end{align}
In addition, $\lim_{\lambda\to+\infty}\Phi(\lambda)=+\infty$ and hence $U(\{0\})=0$.
Moreover, we can obtain an explicit formula of $u(x)$. See Remark \ref{Rem:double-stable}.
\end{Ex}

\begin{Ex}\label{Ex2:double-stable-infty}
Under the setting of Example \ref{Ex:double-stable-infty},
the tail of the corresponding L\'evy measure is given by
\begin{align}\label{Ex:double-stable-levy}
  \Pi((x,+\infty))
  =
  \mathbbm{1}_{\{\alpha>0\}}\frac{C_1x^{-\alpha}}{\Gamma(1-\alpha)}
  +\frac{C_2 x^{-\beta}}{\Gamma(1-\beta)},
  \quad
  0<x<+\infty.	
\end{align}
If $\beta=1$, then the second term in the right-hand side is zero. 
Theorem \ref{Thm:conv-DL-infty} yields that,
for $s\in(0,+\infty)$, the finite measure 
\begin{align}
\bP\bigg(\frac{X_{T(s)-}}{s}\in dx,\; \frac{X_{T(s)-}}{s}<1\bigg),
\quad 0<x<1,
\end{align}
(which is a probability measure if $0<\alpha<\beta<1$) admits a density 
\begin{align}
   s\Pi(&(s-sx,+\infty))u(sx)
   \\
   =\,&
   \frac{\sin(\pi\alpha)}{\pi}x^{-1+\alpha}(1-x)^{-\alpha}
   \\
   &+
    \mathbbm{1}_{\{\alpha>0\}}\sum_{k=1}^n 
   \bigg(-\frac{C_2}{C_1}\bigg)^k
   \frac{s^{-k(\beta-\alpha)}}{\Gamma(\alpha-k(\beta-\alpha))\Gamma(1-\alpha)}
   x^{-1+\alpha-k(\beta-\alpha)}(1-x)^{-\alpha}
   \\
   &-
   \sum_{k=0}^{n-1}
   \bigg(-\frac{C_2}{C_1}\bigg)^{k+1}
   \frac{s^{-(k+1)(\beta-\alpha)}}{\Gamma(\alpha-k(\beta-\alpha))\Gamma(1-\beta)}
   x^{-1+\alpha-k(\beta-\alpha)}
   (1-x)^{-\beta}
   \\
   &+
   O(s^{-(n+1)(\beta-\alpha)}),
   \quad
   \text{as $s\to+\infty$,}
   \label{Ex-1}	
\end{align}
uniformly in $x\in I$ for any closed interval $I\subset (0,1)$. 
We can also obtain a similar asymptotic expansion of $\bP(X_{T(s)-}/s\in I)$ as $s\to+\infty$.
\end{Ex}

Next, we will consider a short-range version of Example \ref{Ex:double-stable-infty}.

\begin{Ex}\label{Ex:double-stable-zero}
Under the setting of Example \ref{Ex:double-stable-infty}, we can also check that the assumption
 of Theorem \ref{Thm:potential-zero} is satisfied, as we shall see below. Note that
\begin{align}
   \frac{1}{\Phi(z)}
   =
   \frac{1}{C_2 z^\beta (C_1z^{-(\beta-\alpha)}/C_2 +1)}
   =
   \frac{1}{C_2}
   \sum_{k=0}^\infty
   	\bigg(-\frac{C_1}{C_2}\bigg)^k
   	z^{-\beta-k(\beta-\alpha)},
\end{align}
for any $z\in\bC$ with ${\rm Re}\,z>0$ and $C_1|z|^{-(\beta-\alpha)}/C_2<1$. By differentiating both sides and multiplying by $-1$, we have
\begin{align}
    \frac{\Phi'(z)}{\Phi(z)^2}
    &=
    \frac{1}{C_2}\sum_{k=0}^\infty
    \bigg(-\frac{C_1}{C_2}\bigg)^k
    (\beta+k(\beta-\alpha))z^{-1-\beta-k(\beta-\alpha)}
    \\
    &=
    \frac{1}{C_2}\sum_{k=0}^n
    \bigg(-\frac{C_1}{C_2}\bigg)^k
    (\beta+k(\beta-\alpha))z^{-1-\beta-k(\beta-\alpha)}
     +
    O(|z|^{-1-\beta-(n+1)(\beta-\alpha)}),
    \\
    &\qquad
    \phantom{{C_2}^{-1}\sum_{k=0}^n
    (-C_1/C_2)^k
    (\beta+k(\beta-\alpha))z}
    \text{as $z\to\infty$ in $\{z\in\bC\mid {\rm Re}\,z>0\}$,}
\end{align}
for any $n\in\bN\cup\{0\}$. Hence the assumption \ref{condition:infty} is satisfied with
\begin{align}
	c_k=C_2\bigg(-\frac{C_2}{C_1}\bigg)^k,
	\quad
	\beta_k=\beta+k(\beta-\alpha)
	\quad
	\text{and}
	\quad
	R(t)
	=t^{-1-\beta-(n+1)(\beta-\alpha)}.
\end{align}
Theorem \ref{Thm:potential-zero} yields that $U|_{(0,+\infty)}(dx)$ admits a continuous density
\begin{align}
   u(x)
   =\frac{1}{C_2}
   \sum_{k=0}^n
   \bigg(-\frac{C_1}{C_2}\bigg)^k
   \frac{x^{-1+\beta+k(\beta-\alpha)}}{\Gamma(\beta+k(\beta-\alpha))}
   +
   O(x^{-1+\beta+(n+1)(\beta-\alpha)}),
   \quad\text{as $x\to0+$.}
\end{align}
Moreover we can obtain an explicit formula of $u(x)$. See Remark \ref{Rem:double-stable}. 
\end{Ex}

\begin{Ex}
Under the setting of Example \ref{Ex:double-stable-infty},  the tail of the corresponding L\'evy measure is given by \eqref{Ex:double-stable-levy}.
Therefore Theorem \ref{Thm:conv-DL-zero} implies that, for $s\in(0,+\infty)$, the finite measure 
\begin{align}
\bP\bigg(\frac{X_{T(s)-}}{s}\in dx,\;\frac{X_{T(s)-}}{s}<1 \bigg),
\quad 0<x<1,
\end{align}
(which is a probability measure if $0<\alpha<\beta<1$) admits a density
\begin{align}
   s\Pi(&(s-sx,+\infty))u(sx)
   \\
   =\,&
   \frac{\sin(\pi\beta)}{\pi}x^{-1+\beta}(1-x)^{-\beta}
   \\
   &+
   \sum_{k=1}^n 
   \bigg(-\frac{C_1}{C_2}\bigg)^k
   \frac{s^{k(\beta-\alpha)}}{\Gamma(\beta+k(\beta-\alpha))\Gamma(1-\beta)}
   x^{-1+\beta+k(\beta-\alpha)}(1-x)^{-\beta}
   \\
   &-
   \mathbbm{1}_{\{\alpha>0\}}\sum_{k=0}^{n-1} \bigg(-\frac{C_1}{C_2}\bigg)^{k+1}\frac{s^{(k+1)(\beta-\alpha)}}{\Gamma(\beta+k(\beta-\alpha))\Gamma(1-\alpha)}
   x^{-1+\beta+k(\beta-\alpha)}(1-x)^{-\alpha}
   \\
   &+
   O(s^{(n+1)(\beta-\alpha)}),
   \quad
   \text{as $s\to0+$,}	
\end{align}
uniformly in $x\in I$ for any closed interval $I\subset (0,1)$. 
We can also obtain a similar asymptotic expansion of $\bP(X_{T(s)-}/s\in I)$ as $s\to0+$.
\end{Ex}

\begin{Rem}
\label{Rem:double-stable}
Under the setting of Example \ref{Ex:double-stable-infty}, we can calculate $u(x)$ explicitly in a similar way as the proof of Theorem \ref{Thm:potential-zero}. Indeed,
for $x\in(0,+\infty)$ with $C_1 x^{(\beta-\alpha)}/C_2<1$, we use the dominated convergence theorem to get
\begin{align}
  	u(x)
  	&=
  	\frac{x^{-1}e}{2\pi}
  	\int_{-\infty}^{+\infty}
  	e^{i\theta x}
  	\frac{\Phi'(x^{-1}+i\theta)}{\Phi(x^{-1}+i\theta)^2}
  	\,d\theta
  	\\
  	&=
  	\frac{1}{C_2}
  	\sum_{k=0}^\infty
  	\bigg(-\frac{C_1}{C_2}\bigg)^k
  	\frac{x^{-1}e}{2\pi}
  	\int_{-\infty}^{+\infty}
  	e^{i\theta x}
  	\frac{\beta+k(\beta-\alpha)}{(x^{-1}+i\theta)^{1+\beta+k(\beta-\alpha)}}d\theta
  	\\
  	&=
  	\frac{1}{C_2}
  	\sum_{k=0}^\infty
  	\bigg(-\frac{C_1}{C_2}\bigg)^k
  	\frac{x^{-1+\beta+k(\beta-\alpha)}}{\Gamma(\beta+k(\beta-\alpha))}
  	\\
  	&= 
  	\frac{x^{-1+\beta}}{C_2}
  	E_{\beta-\alpha,\beta}
  	\bigg(-\frac{C_1 x^{\beta-\alpha}}{C_2}\bigg),
  	\label{Rem:double-stable-1}
\end{align}
where $E_{a,b}(z)$ denotes the two-parameter Mittag-Leffler function
\begin{align}
	E_{a,b}(z)
	=
	\sum_{k=0}^\infty
	\frac{z^k}{\Gamma(ka+b)},
	\quad
	z\in\bC,\,a,b>0.
\end{align} 
We remark that the condition for the dominated convergence theorem can be checked by a similar calculation as in \eqref{Proof-14}. 
We use \eqref{u(x)} with $N=2$ to see that $u(x)$ can be extended to an analytic function on $\{z\in\bC\mid {\rm Re}\, z>0\}$.
In addition,  the right-hand side of \eqref{Rem:double-stable-1} can also be extended to an analytic function on $\{z\in\bC\mid {\rm Re}\, z>0\}$. 
Therefore we use the identity theorem for analytic functions to obtain
\begin{align}
	u(x)
	=
  	\frac{x^{-1+\beta}}{C_2}
  	E_{\beta-\alpha,\beta}
  	\bigg(-\frac{C_1 x^{\beta-\alpha}}{C_2}\bigg),
  	\quad
  	0<x<+\infty,
  	\label{Rem:double-stable-2}
\end{align}
which recovers \cite[Example 5.25]{Kyp}. Combining \eqref{Rem:double-stable-2} with the asymptotic expansion (see \cite[Section 18.1, (21)]{EMOT} for example)
\begin{align}
	E_{a,b}(-x)
	=
	\sum_{k=1}^{n+1}
	\frac{(-1)^{k-1}x^{-k}}{\Gamma(-ka+b)}+O(x^{-(n+2)}),
	\quad
	\text{as $x\to+\infty$, $0<a<2$, $n\in\bN\cup\{0\}$, }
\end{align}
we obtain \eqref{Ex:double-stable-infty-1} again.

\end{Rem}

Next, let us focus on geometric stable subordinators. 

\begin{Ex}\label{Ex:geometric}
Let us consider the case 
\begin{align}
\Phi(z)=\log(1+z^\alpha)
\end{align}
for some $0<\alpha\leq 1$. The corresponding subordinator $X$ is called a geometric $\alpha$-stable subordinator, which was studied in \cite[Section 2]{SikSonVon} and \cite[Example 5.11]{BBKRSV}, for example.  If $\alpha=1$, then $X$ is called a gamma subordinator \cite[Example 5.10]{BBKRSV}.
In the following we will check that the assumptions of Theorem \ref{Thm:potential-infty} are satisfied.
We denote by
\begin{align}
    \frac{1}{\log(1+z)}
   =
   \sum_{k=-1}^\infty b_k z^k,
   \quad
   \text{$z\in\bC$ with $0<|z|<1$,}
   \label{Taylor-z-log(1+z)} 	
\end{align}
the Laurent series expansion of $1/\log(1+z)$.
Then the coefficients $b_{-1}, b_0, b_1,\ldots$ are inductively determined by the equations
\begin{align}
	b_{-1}=1,\quad
	\sum_{k=-1}^{n-1} \frac{(-1)^{n-k-1}}{n-k} b_{k}=0,
	\quad
	n\in\bN.
	\label{coefficient}
\end{align}
See Remark \ref{Rem:Taylor}. For example, $b_0=1/2$, $b_1=-1/12$, $b_2=1/24$ and $b_3=-19/720$.
For $z\in\bC$ with ${\rm Re}\,z>0$ and $|z|<1$, we have
\begin{align}
   \frac{1}{\Phi(z)}
   =
   \frac{1}{\log(1+z^\alpha)}
   =
   \sum_{k=-1}^\infty
   b_k z^{k\alpha}.	
\end{align}
Fix $n\in\bN\cup\{0\}$ arbitrarily. Take $N\in\bN$ large enough so that $-N+n\alpha<-1$. Then
\begin{align}
   \bigg(
   \frac{d}{dz}
   \bigg)^N
   \bigg(
   \frac{1}{\Phi(z)}
   -
   \sum_{k=-1}^{n-1} b_k z^{k\alpha}
   \bigg)
   =
   O(|z|^{-N+n\alpha}),
   \\
   \text{as $z\to0$ in $\{z\in\bC\mid {\rm Re}\,z>0\}$.}	
\end{align}
Hence the condition \ref{condition:zero} is satisfied with
\begin{align}
	c_k=1/b_{k-1},
	\quad
	\alpha_k=-(k-1)\alpha
	\quad\text{and}\quad
	R(t)=t^{-N+n\alpha}.
\end{align}
In addition, there exist polynomials $p_{k,N}(z^\alpha)$  $(k=1,2,\dots,N)$ of degree at most $N$ in $z^\alpha$ such that
\begin{align}
  \bigg(\frac{d}{dz}\bigg)^N
  \frac{1}{\Phi(z)}
  =
  \sum_{k=1}^N
  \frac{p_{k,N}(z^\alpha)}
       {(\log(1+z^\alpha))^{k+1}  (z+z^{\alpha+1})^N}.	
\end{align}
Hence, for any $r>0$, there exists a constant $C=C(r,N)>0$ such that
\begin{align}
   \bigg|
   \bigg(\frac{d}{dz}\bigg)^N
  \frac{1}{\Phi(z)}
  \bigg|
  \leq
  C|z|^{-N},
  \quad
  \text{on $\{z\in\bC\mid {\rm Re}\, z>0\,\,\text{and}\,\, |{\rm Im}\,z|>r\}$}. 	
\end{align}
Since $N>1+n\alpha\geq1$, we have
\begin{align}
		\sup_{\lambda>0}\int_{\bR\setminus(-r,r)}
   \bigg|
   \bigg(
   \bigg(\frac{d}{dz}\bigg)^N 
   \frac{1}{\Phi}\bigg)(\lambda+i\theta)\bigg|\,d\theta
   	\leq
   	C
   	\int_{\bR\setminus(-r,r)}
   	|\theta|^{-N}
   	\,d\theta
   	<+\infty.
\end{align}
Therefore the condition \ref{condition:bdd} is satisfied. 
We use Theorem \ref{Thm:potential-infty} to see that $U|_{(0,+\infty)}(dx)$ admits a continuous density 
\begin{align}
	u(x)
	=
	\sum_{k=-1}^{n-1} \frac{b_k x^{-1-k\alpha}}{\Gamma(-k\alpha)}
	+
	O(x^{-1-n\alpha}),
	\quad
	\text{as $x\to+\infty$.}
\end{align}
In addition, $\lim_{\lambda\to+\infty}\Phi(\lambda)=+\infty$ and hence $U(\{0\})=0$.
\end{Ex}

\begin{Ex}\label{Ex2:geometric}
Under the setting of Example \ref{Ex:geometric},
the corresponding L\'evy measure is given by
\begin{align}
  \Pi(dx)
  =
  \frac{\alpha E_\alpha(-x^\alpha)}{x}\,dx.
  \quad
  0<x<+\infty,	
\end{align}
where $E_\alpha(z)$ denotes the Mittag-Leffler function with parameter $\alpha$:
\begin{align}\label{ML}
 	E_\alpha(z)
 	=
 	\sum_{k=0}^\infty
 	\frac{z^k}{\Gamma(k\alpha+1)},
 	\quad
 	z\in\bC.
\end{align}
See \cite[Example 5.11]{BBKRSV} for the details.
Note that our definition of $E_\alpha(x)$ is based on \cite{EMOT}, and is slightly different from the definition in \cite[Example 5.11]{BBKRSV}.
Fix $n\in\bN\cup\{0\}$ arbitrarily.
By using the asymptotic expansion (see for example \cite[Section 18.1, (7)]{EMOT})
\begin{align}
   E_\alpha(-x)
   =
   \sum_{\ell=1}^{n+1} \frac{(-1)^{\ell-1} x^{-\ell}}{\Gamma(-\ell\alpha+1)}
   +O(x^{-(n+2)}),
   \quad \text{as $x\to+\infty$,}	
\end{align}
we have
\begin{align}
  \Pi((x,+\infty))
  &=
  \int_{x}^{+\infty}
  \frac{\alpha E_\alpha(-t^\alpha)}{t}\,dt
  \\
  &=	
  \sum_{\ell=1}^{n+1}
  \frac{(-1)^{\ell-1}x^{-\ell\alpha}}{\ell\Gamma(-\ell\alpha+1)}+O(x^{-(n+2)\alpha}),
  \quad
  \text{as $x\to+\infty$.}
\end{align}
Then Theorem \ref{Thm:conv-DL-infty} yields that the probability measure
\begin{align}
	\bP\bigg(\frac{X_{T(s)-}}{s}\in dx\bigg),
\quad 0<x<1,
\end{align}
admits a density
\begin{align}
   s\Pi(&(s-sx,+\infty))u(sx)
   \\
   =\,&
   \frac{\sin(\pi\alpha)}{\pi}x^{-1+\alpha}(1-x)^{-\alpha}
   \\
   &+
  \sum_{k=-1}^{n-1} \sum_{\ell=1}^{n+1}  \mathbbm{1}_{\{1\leq k+\ell \leq n \}}
   \frac{(-1)^{\ell-1} b_k s^{-(k+\ell)\alpha}}{\ell\Gamma(-k\alpha)  \Gamma(-\ell \alpha+1)}x^{-1-k\alpha}(1-x)^{-\ell\alpha}
   \\
   &+O(s^{-(n+1)\alpha}),
   \quad
   \text{as $s\to+\infty$,}	
\end{align}
uniformly in $x\in I$ for any closed interval $I\subset (0,1)$. 
We can also obtain a similar asymptotic expansion of $\bP(X_{T(s)-}/s\in I)$ as $s\to+\infty$.

\end{Ex}

\begin{Rem}\label{Rem:Taylor}
Let us prove \eqref{coefficient}.
Recall that
\begin{align}
   \log (1+z)
   =
   \sum_{n=1}^\infty \frac{(-1)^{n-1}}{n} z^n,
   \quad
   \text{$z\in\bC$ with $|z|<1$},	
\end{align}
and hence
\begin{align}
   1=\log (1+z)\frac{1}{\log(1+z)}
    =
    \sum_{n=0}^\infty\bigg(\sum_{k=-1}^{n-1} \frac{(-1)^{n-k-1}}{n-k} b_{k}\bigg)z^n,	
\end{align}
which implies \eqref{coefficient}. 	
\end{Rem}

Finally, we will consider the independent sum of a stable subordinator and a geometric stable subordinator.

\begin{Ex}\label{Ex:sum-stable-gamma}
Let us consider the case
\begin{align}
   \Phi(z)=z^\beta+\log(1+z^\alpha)	
\end{align}
for some $\alpha, \beta\in(0,1]$. Let us show that the assumption of Theorem \ref{Thm:potential-zero} is satisfied. We have
\begin{align}
  \frac{1}{\Phi(z)}
  =
  \frac{1}{z^\beta (1+z^{-\beta}\log(1+z^\alpha))}
  =
  \sum_{k=0}^\infty z^{-(k+1)\beta}(-\log(1+z^\alpha))^{k}	
\end{align}
for any $z\in\bC$ with ${\rm Re}\,z>0$ and $|z^{-\beta}\log(1+z^\alpha)|<1$. Hence
\begin{align}
  \bigg(-\frac{d}{dz}\bigg)&
  \bigg(
  \frac{1}{\Phi(z)}
  -
  \frac{1}{z^\beta}
  \bigg)
  \\
  &=
  \sum_{k=1}^\infty
  z^{-(k+1)\beta}(-\log(1+z^\alpha))^{k}
  \bigg(\frac{(k+1)\beta}{z}-\frac{k\alpha z^{\alpha-1}}{(1+z^\alpha)\log(1+z^\alpha)}\bigg)
  \\
  &=
  O(|z|^{-1-2\beta}\log|z|),
  \quad
  \text{as $z\to\infty$ in $\{z\in\bC\mid {\rm Re}\,z>0\}$.}	
\end{align}
Therefore the condition \ref{condition:infty} is satisfied.	We use Theorem \ref{Thm:potential-zero} to see that $U|_{(0,+\infty)}(dx)$ admits a continuous density
\begin{align}
   u(x)
   =
   \frac{x^{-1+\beta}}{\Gamma(\beta)}+O(x^{-1+2\beta}
   |\log x|),
   \quad
   \text{as $x\to 0+$.}	
\end{align}
In addition, $\lim_{\lambda\to+\infty}\Phi(\lambda)=+\infty$ and hence $U(\{0\})=0$. 
\end{Ex}

\begin{Ex}
Under the setting of Example \ref{Ex:sum-stable-gamma}, the tail of the corresponding L\'evy measure is given by
\begin{align}
   \Pi((x,+\infty))
   &=
   \frac{x^{-\beta}}{\Gamma(1-\beta)}
   +
   \int_x^{+\infty}
   \frac{\alpha E_\alpha(-t^\alpha)}{t}
   \,dt
   \\
   &=
   \frac{x^{-\beta}}{\Gamma(1-\beta)}
   +O(|\log x|),
   \quad
   \text{as $x\to0+$,}	
\end{align}
where $E_\alpha(z)$ is the Mittag-Leffler function given by \eqref{ML}.
By Theorem \ref{Thm:conv-DL-zero},  for any $s\in(0,+\infty)$, we see that the probability measure
\begin{align}
	\bP\bigg(\frac{X_{T(s)-}}{s}\in dx\bigg),
\quad 0<x<1,
\end{align}
admits a density
\begin{align}
   s\Pi(&(s-sx, +\infty))u(sx)
   \\
   &=
   \frac{\sin(\pi\beta)}{\pi}
   x^{-1+\beta}(1-x)^{-\beta}
   +
   O(s^{\beta}|\log s|),
   \quad
  \text{as $s\to0+$,}	
\end{align}
uniformly in $x\in I$ for any closed interval $I\subset (0,1)$. 
We can also obtain a similar asymptotic expansion of $\bP(X_{T(s)-}/s\in I)$ as $s\to0+$.	
\end{Ex}

\subsection*{Acknowledgements}
This research was partially supported by JSPS KAKENHI Grant Number JP24K16948, by the ISM Cooperative Research Symposium (2024-ISMCRP-5003) and by the Research Institute for Mathematical Sciences, an International Joint Usage/Research Center located in Kyoto University.

\newcommand{\etalchar}[1]{$^{#1}$}

\end{document}